\documentclass[a4paper, 11pt]{article}

\usepackage[latin1]{inputenc}
\usepackage[english]{babel}
\usepackage{graphicx}
\usepackage{amssymb}
\usepackage{amsmath}
\usepackage{amsfonts}
\usepackage{amsthm}
\usepackage{url}

\newtheorem{theorem}{Theorem}

\newtheorem{corollary}{Corollary}

\newtheorem{example}{Example}
\newtheorem{lemma}{Lemma}

\newtheorem{remark}{Remark}
\numberwithin{equation}{section}
\DeclareMathOperator{\argmin}{argmin}
\DeclareMathOperator{\rank}{rank}

\title{Low-Rank Matrix Approximation \\ with Weights or Missing Data is NP-hard}

\author{\normalsize Nicolas Gillis${}^1$ and Fran\c{c}ois Glineur${}^1$}

\date{}

\begin{document}
\renewcommand{\labelitemi}{$\diamond$}

\maketitle

\begin{abstract}
Weighted low-rank approximation (WLRA), a dimensionality reduction technique for data analysis, has been successfully used in several applications, such as in collaborative filtering to design recommender systems or in computer vision to recover structure from motion. 
In this paper, we prove that computing an optimal weighted low-rank approximation is NP-hard, already when a rank-one approximation is sought. In fact, we show that it is hard to compute approximate solutions to the WLRA problem with some prescribed accuracy. 
Our proofs are based on reductions  from the maximum-edge biclique problem, and apply to strictly positive weights as well as to binary weights (the latter corresponding to low-rank matrix approximation with missing data).  
\bigskip

\end{abstract}

\footnotetext[1] {Universit\'e catholique de Louvain, CORE, B-1348 Louvain-la-Neuve, Belgium. 
E-mail: nicolas.gillis@uclouvain.be and francois.\mbox{glineur}@uclouvain.be.  Nicolas Gillis is a research fellow of the Fonds de la Recherche Scientifique (F.R.S.-FNRS).  This text presents research results of the Belgian Program on Interuniversity Poles of Attraction initiated by the Belgian State, Prime Minister's Office, Science Policy
Programming. The scientific responsibility is assumed by the authors.}

\section{Introduction} 

Approximating a matrix with one of lower rank is a key problem in data analysis and 
is widely used for linear dimensionality reduction. 
Numerous variants exist emphasizing different constraints and objective functions, e.g., principal component analysis (PCA)~\cite{J86}, independent component analysis~\cite{C94}, nonnegative matrix factorization~\cite{LS1}, and other refinements are often imposed on these models, e.g., sparsity to improve interpretability or increase compression~\cite{AE07}.

In some cases, it may be necessary to attach a weight to each entry of the data matrix, expressing its relative importance \cite{GZ79}. This is for example the case in the following situations:
\begin{itemize}

\item The matrix to be approximated is obtained via a sampling procedure and the number of samples and/or the expected variance vary among the entries. For example, it has been shown that using a weighted norm gives better results in 2-D digital filter design \cite{LPW97} and microarray data analysis \cite{MN09}.  



\item Some data is missing/unknown, which can be taken into account by assigning zero weights to the missing/unknown entries of the data matrix. 
This is for example the case in collaborative filtering, notably used to design recommender systems \cite{Sar} (in particular, the Netflix prize competition has demonstrated the effectiveness of low-rank matrix factorization techniques \cite{KBV09}), or in computer vision to recover structure from motion \cite{SIR95, J01}, see also \cite{C08}. This problem is often referred to as \emph{PCA with missing data} \cite{SIR95, GM98}, and can be viewed as a  \emph{low-rank matrix completion problem with noise}, i.e., approximate a given noisy data matrix featuring missing entries with a low-rank matrix\footnote{In our settings, the rank of the approximation is  fixed a priori.}. 

\item  A greater emphasis must be placed on the accuracy of the approximation on 
a localized part of the data, a situation encountered for example in image processing \cite[Chapter 6]{diep}. \\ 
\end{itemize}

Finding a low-rank matrix which is closest to the input matrix according to these weights is an optimization problem called \emph{weighted low-rank approximation} (WLRA). Formally, it can be formulated as follows: first, given an $m$-by-$n$ nonnegative weight matrix $W \in \mathbb{R}^{m \times n}_+$, we define the weighted Frobenius norm
of  an $m$-by-$n$ matrix $A$ as  $|| A ||_W = (\sum_{i,j} W_{ij} A_{ij}^2)^\frac12$. 
Then, given an $m$-by-$n$ real matrix $M \in \mathbb{R}^{m \times n}$ and a positive integer $r \leq \min(m,n)$, we seek an $m$-by-$n$ matrix ${X}$ with rank at most $r$ that approximates $M$ as closely as possible, where the quality of the approximation is measured by the weighted Frobenius norm of the error: 
\[ 
p^* = \inf_{X  \in \mathbb{R}^{m \times n}} ||M - X||_W^2 \text{ such that $X$ has rank at most $r$}. 
\]
Since any $m$-by-$n$ matrix with rank at most $r$ can be expressed as the product of two matrices of dimensions $m$-by-$r$ and $r$-by-$n$, we will use the following more convenient formulation featuring two unknown matrices $U \in \mathbb{R}^{m \times r}$ and $V \in \mathbb{R}^{n \times r}$ but no explicit rank constraint:
\begin{equation} 
p^* = \inf_{
U  \in \mathbb{R}^{m \times r}, V \in \mathbb{R}^{n \times r}} \quad
||M-UV^T||_W^2 = \sum_{ij} W_{ij} (M-UV^T)_{ij}^2\;. \label{WLRA} \tag{WLRA}\\
\nonumber
\end{equation}
Even though \eqref{WLRA} is suspected to be NP-hard \cite{J01, SJ04}, this has never, to the best of our knowledge, been studied formally. In this paper, we analyze the computational complexity in the rank-one case (i.e., for $r=1$) and prove the following two results. 

\begin{theorem} \label{WLRAnphard}
When $M \in \{0,1\}^{m \times n}$, and $W \in \; ]0,1]^{m \times n}$, it is NP-hard to find an approximate solution of rank-one \eqref{WLRA} with objective function accuracy less than $2^{-11}(mn)^{-6}$. 
\end{theorem}

\begin{theorem} \label{LRAMDnphard} 
When $M \in [0,1]^{m \times n}$, and $W \in \{0,1\}^{m \times n}$, it is NP-hard to find an approximate solution of rank-one \eqref{WLRA} with objective function accuracy less than $2^{-12}(mn)^{-7}$. \\
\end{theorem} 
In other words, it is NP-hard to find an approximate solution to 
rank-one \eqref{WLRA} with positive weights, and to 
the rank-one matrix approximation problem with missing data. 
Note that these results can be easily generalized to any fixed rank $r$, see Remark~\ref{Remnmu}.

The paper is organized as follows. 
We first review existing results about the complexity of \eqref{WLRA} in Section~\ref{formu}. 
In Section~\ref{MBPsec}, we introduce the maximum-edge biclique problem (MBP), which is NP-hard.  
In Sections~\ref{WLRAcomp} and \ref{LRMD}, we prove Theorems~\ref{WLRAnphard} and \ref{LRAMDnphard} respectively, using polynomial-time reductions from MBP. We conclude with a discussion and some open questions. \\



\subsection{Notation} 
The set of real $m$-by-$n$ matrices is denoted $\mathbb{R}^{m \times n}$, or $\mathbb{R}^{m \times n}_+$ when all the entries are required to be nonnegative. 
For $A \in \mathbb{R}^{m \times n}$, 
we note $A_{:j}$ the $j^{\text{th}}$ column of $A$, 
$A_{i:}$ the $i^{\text{th}}$ row of $A$,  
and $A_{ij}$ or $A(i,j)$ the entry at position $(i,j)$; 
for $b \in \mathbb{R}^{m \times 1} = \mathbb{R}^{m}$, 
we note $b_i$ 
the $i^{\text{th}}$ entry of $b$.  
The transpose of $A$ is $A^T$. 
The Frobenius norm of a matrix $A$ is defined as $||A||_F^2 = \sum_{i,j} (A_{ij})^2$, and $||.||_2$ is the usual Euclidean norm  with $||b||_2^2 = {\sum_{i} b_i^2}$.  For $W \in \mathbb{R}^{m \times n}_+$, the weighted Frobenius `norm' of a matrix $A$ is  defined\footnote{$||.||_W$ is a matrix norm if and only if $W > 0$, else it is a semi-norm.} 
by $||A||_W^2 = \sum_{i,j} W_{ij} (A_{ij})^2$. 
The $m$-by-$n$ matrix of all ones is denoted $\mathbf{1}_{m \times n}$, the $m$-by-$n$ matrix of all zeros  $\mathbf{0}_{m \times n}$, and $I_{n}$ is the identity matrix of dimension $n$. The smallest integer larger or equal to $x$ is denoted $\lceil x \rceil$.

\section{Previous Results} \label{formu}


Weighted low-rank approximation is suspected to be much more difficult than the corresponding unweighted problem (i.e., when $W$ is the matrix of all ones), which 
is efficiently solved using the singular value decomposition (SVD) \cite{Gol}.  
In fact, it has been previously observed that the weighted problem might have several local minima which are not global \cite{SJ04}, while this cannot occur in the unweighted case (i.e., when $W$ is the matrix of all ones), see, e.g., \cite[p.29, Th.1.14]{diep}. 
\begin{example} \label{examp} Let 
\begin{displaymath}
M = \left( \begin{array}{ccc}
     1  &   0 &    1 \\
     0    & 1  &   1\\
     1   &  1  &   1\\\end{array} \right), \quad \textrm{ and }  \quad
W = \left( \begin{array}{ccc}
     1  &   100 &    2 \\
     100   & 1  &   2\\
     1   &  1  &   1\\ \end{array} \right).
\end{displaymath}
In the case of a rank-one factorization ($r=1$) and a nonnegative matrix $M$, one can impose without loss of generality that the solutions of \eqref{WLRA} are nonnegative. 
In fact, one can easily check that any rank-one solution $uv^T$ of \eqref{WLRA} can only be improved by taking its component-wise absolute value $|uv^T| = |u||v|^T$. Moreover, we can impose without loss of generality that $||u||_2 = 1$, so that only two degrees of freedom remain. Indeed,  for a given 
\[
u(x,y) = \left( \begin{array}{c}
     x   \\
     y    \\
     \sqrt{1-x^2-y^2} \\ \end{array} \right), 
     \; \textrm{ with } \,
     \left\{ \begin{array}{l}
x \geq 0, y \geq 0 \\
x^2 +  y^2   \leq 1
\end{array} \right. ,
\] 
the corresponding optimal $v^*(x,y) = \argmin_{v} ||M-u(x,y)v^T||_W^2$ can be computed easily\footnote{This problem can be decoupled into $n$ independent quadratic programs in one variable, and admits the following closed-form solution: $v^*(x,y) = [(M \circ W)^T u] /. [W^T (u \circ u)]$, where $\circ$ (resp.\@ $/.$) is the component-wise multiplication (resp.\@ division).}. 
Figure \ref{localmin} displays the graph of the objective function $||M-u(x,y) v^*(x,y)^T||_W$ with respect to parameters $x$ and $y$; we observe four local minima, close to $(\frac{\sqrt{2}}{2},0)$, $(0,\frac{\sqrt{2}}{2})$, $(0,0)$ and $(\frac{\sqrt{2}}{2},\frac{\sqrt{2}}{2})$.
\begin{figure}[ht!]
\begin{center}
\includegraphics[width=12cm]{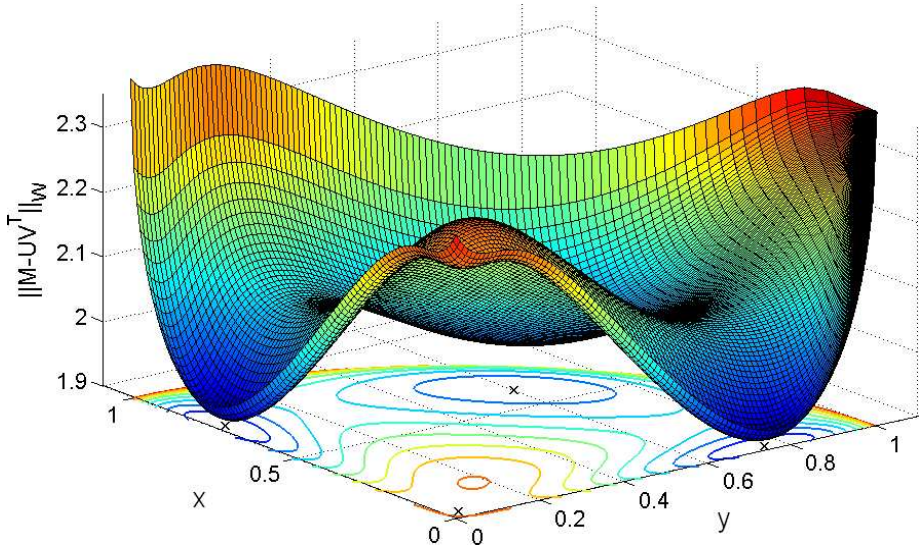}
\caption{Objective function of \eqref{WLRA} with respect to the parameters $(x,y)$.}
\label{localmin}
\end{center}
\end{figure}
We will see later in Section~\ref{complex} how this example has been generated. \vspace{0.2cm}
\end{example}

However, if the rank of the weight matrix $W \in \mathbb{R}^{m \times n}_+$ is equal to one, i.e., $W = st^T$ for some $s \in \mathbb{R}^m_+$ and $t \in \mathbb{R}^n_+$,  \eqref{WLRA} can be reduced to an unweighted low-rank approximation. In fact,
\begin{eqnarray*}
||M-UV^T||_W^2
& = & \sum_{i,j} W_{ij} \, (M-UV^T)_{ij}^2  = \sum_{i,j} s_i t_j \, (M-UV^T)_{ij}^2 \\
& = & \sum_{i,j}
\Big( \sqrt{s_it_j} \, M_{ij} - (\sqrt{s_i}\,U_{i:})(\sqrt{t_j}\,V_{j:}^T)\Big)^2. 
\end{eqnarray*}
Therefore, if we define a matrix $M'$ such that $M'_{ij} = \sqrt{s_it_j} \, M_{ij}$ $\forall i,j$, an optimal weighted low-rank approximation $(U,V)$ of $M$ can be recovered from a solution $(U',V')$ to the unweighted problem for matrix $M'$ using $U_{i:}= U'_{i:}/\sqrt{s_i}$ $\forall i$ and $V_{j:}= V'_{j:}/\sqrt{t_j}$ $\forall j$.\\



When the weight matrix $W$ is binary, WLRA amounts to approximating a matrix with missing data.
This problem is closely related to 
 \emph{low-rank matrix completion}, see~\cite{CR09} and the references therein, which can be defined as
\begin{equation}
\min_{X} \; \rank(X) \quad 
\text{ such that } 
X_{ij} = M_{ij} \text{ for } (i,j) \in \Omega, \label{MC} \tag{MC}
\end{equation}
where $\Omega \subseteq \{1,2,\dots,m\} \times \{1,2,\dots,n\}$ is the set of entries for which the values of $M$ are known. 
\eqref{MC} has been shown to be  NP-hard \cite{CG84}, and it is clear that an optimal solution $X^*$ of \eqref{MC} can be obtained by solving a sequence of \eqref{WLRA} problems with the same matrix $M$, with 
\[
W_{ij} = \left\{ 
\begin{array}{ll}
1 & \textrm{if } (i,j) \in \Omega \\
0 & \textrm{otherwise}
\end{array} \right. ,
\]
and for different values of the target rank ranging from $r = 1$ to $r = \min(m,n)$. The smallest value of $r$ for which the objective function $||M-UV^T||_W^2$ of \eqref{WLRA} vanishes provides an optimal solution for \eqref{MC}. 
This observation implies that it is NP-hard to solve \eqref{WLRA} for each possible value of $r$ from $1$ to $\min(m,n)$, since it would solve \eqref{MC}. However, this does not imply that \eqref{WLRA} is NP-hard when $r$ is fixed, and in particular when $r$ equals one. In fact, checking whether \eqref{MC} admits a rank-one solution can be done easily\footnote{The solution $X = uv^T$ can be constructed  observing that the vector $u$ must be a multiple of each column of $M$.}. 

Rank-one \eqref{WLRA} can be equivalently reformulated as 
\[
\inf_{A} ||M-A||_W^2 \quad \text{ such that } \quad \rank(A) \leq 1, 
\]
and, when $W$ is binary, is the problem of finding, if possible, the best rank-one approximation of a matrix with missing entries. 
To the best of our knowledge, the complexity of this problem has never been studied  formally; it will be shown to be NP-hard in the next section. \\

Another closely related result is the NP-hardness of the structure from motion problem (SFM), in the presence of noise and missing data \cite{NKS07}. Several points of a rigid object are tracked with cameras (we are given the projections of the 3-D points on the 2-D camera planes)\footnote{Missing data arise because the points may not always be visible by the cameras, e.g., in the case of a rotation.}, 
and the aim is to recover the structure of the object and the positions of the 3-D points. SFM can be written as a rank-four \eqref{WLRA} problem with a binary weight matrix\footnote{With the additional constraint that the last row of $V$ must be all ones, i.e., $V_{r:} = \mathbf{1}_{1 \times n}$.} \cite{J01}. 
However, this result does not imply anything on the complexity of rank-one \eqref{WLRA}. \\



An important feature of \eqref{WLRA} is exposed by the following example. 
\begin{example} \label{ex2}
Let 
\[
M = \left( \begin{array}{cc} 1 & ? \\ 0 & 1 \end{array} \right), 
\]
where ? indicates that an entry is missing, i.e., that the weight associated with this entry is 0 (1 otherwise). Observe that $\forall (u,v) \in \mathbb{R}^m \times \mathbb{R}^n$,
\[
\textrm{rank}(M) = 2 \; \textrm{ and } \; \textrm{rank}(uv^T) = 1 \quad \Rightarrow \quad ||M-uv^T||_W > 0.
\]
However, we have
\[
\inf_{(u,v) \in \mathbb{R}^m \times \mathbb{R}^n} ||M-uv^T||_W = 0.
\]
In fact, one can check that with
\[ 
u(\epsilon) = \left( \begin{array}{c} 1 \\ \epsilon \end{array} \right) \textrm{ and } \;
v(\epsilon) = \left( \begin{array}{c} 1 \\ \epsilon^{-1} \end{array} \right), \textrm{ we have }
\lim_{\epsilon \rightarrow 0} ||M-u(\epsilon)v(\epsilon)^T||_W = 0.
\]
\end{example} 
This indicates that, when $W$ has zero entries, the set of optimal solutions of \eqref{WLRA} might be empty. In other words, the (bounded) infimum of the objective function might be unattained. 
On the other hand, the infimum is always attained for $W > 0$ since $||.||_W$ is then a norm. 

For this reason, these two cases will be analyzed separately: in Section~\ref{WLRAcomp}, we study the computational complexity of the problem when $W > 0$, and, in Section~\ref{LRMD},  the case of a binary $W$ (i.e., the problem with missing data).


\section{Complexity of rank-one \eqref{WLRA}} \label{complex}

In this section, we use  polynomial-time reductions from the maximum-edge biclique problem to prove Theorems~\ref{WLRAnphard} and \ref{LRAMDnphard}.

\subsection{Maximum-Edge Biclique Problem}  \label{MBPsec}

A \emph{bipartite graph} is a graph whose vertices can be partitioned into two disjoint sets such that there is no edge between two vertices in the same set. The maximum-edge biclique problem (MBP) in a bipartite graph is the problem of finding a complete bipartite subgraph (a \emph{biclique}) with the maximum number of edges. 

Let $M \in \{ 0,1 \}^{m \times n}$ be the biadjacency matrix of a bipartite graph 
$G_b = (V_1 \cup V_2,E)$ with $V_1 = \{s_1, \dots s_m\}$,  
$V_2 = \{t_1, \dots t_n\}$ and $E \subseteq (V_1 \times V_2)$ , i.e., 
\[
M_{ij} = 1 \quad \iff \quad (s_i,t_j) \in E.
\]
The cardinality of $E$ will be denoted $|E| = ||M||_F^2 \leq mn$. \\



For example, Figure~\ref{graph} displays the graph $G_b$ generated by the matrix $M$ of Example~\ref{examp}.  
\begin{figure}[ht!]
\begin{center}
\includegraphics[width=3cm]{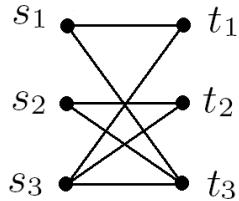}
\caption{Graph corresponding to the matrix $M$ of Example~\ref{examp}.}
\label{graph}
\end{center}
\end{figure}

\noindent With this notation, the maximum-edge biclique problem in a bipartite graph can be formulated as follows \cite{GG09} 
\begin{align}
 \min_{u, v} \qquad & ||M-uv^T||_F^2 \nonumber \\
&  u_i v_j \leq M_{ij}, \; \forall \, i ,j \label{MBP}  \tag{MBP} \\
&  u \in \{ 0,1 \}^{m}, v \in \{ 0,1 \}^{n},  \nonumber
\end{align}
where $u_i = 1$ (resp.\@ $v_j = 1$) means that node $s_i$ (resp.\@ $t_j$) belongs to the solution, $u_i = 0$ (resp.\@ $v_j = 0$) otherwise. The first constraint guarantees feasible solutions of \eqref{MBP} to be bicliques of $G_b$. In fact, it is equivalent to the implication 
\[
M_{ij} = 0 \quad \Rightarrow \quad u_i = 0 \; \textrm{ or } \; v_j = 0,
\]
i.e., if there is no edge between vertices $s_i$ and $t_j$, they cannot simultaneously belong to a solution. 
The objective function minimizes the number of edges outside the biclique, which is equivalent to maximizing the number of edges inside the biclique. Notice that the minimum of \eqref{MBP} is $|E|-|E^*|$, where $|E^*|$ denotes the number of edges in an optimal biclique. 

The decision version of the MBP problem: 
\begin{quote}
\textit{Given $K$, does $G_b$ contain a biclique with at least $K$ edges?} 
\end{quote}
has been shown to be NP-complete \cite{Peet} in the usual Turing machine model \cite{Gar}, which is our framework in this paper. Therefore, computing $|E|-|E^*|$, the optimal value of \eqref{MBP}, is NP-hard.

\subsection{Low-Rank Matrix Approximation with Positive Weights} \label{WLRAcomp}

In order to prove NP-hardness of rank-one \eqref{WLRA} with positive weights ($W > 0$), let us consider the  following instance:  
\begin{equation} \tag{W-1d}  \label{WLRA1} 
p^* = \min_{u \in \mathbb{R}^{m}, v \in \mathbb{R}^{n}}  ||M-uv^T||_{W}^2, 
\end{equation}
with $M \in \{ 0,1 \}^{m \times n}$ the biadjacency matrix of a bipartite graph $G_b = (V,E)$  
and the weight matrix defined as
\[
W_{ij} = \left\{ 
\begin{array}{ll}
1 & \textrm{if } M_{ij} = 1 \\
d & \textrm{if } M_{ij} = 0
\end{array} \right. , \; 1 \leq i \leq m, 1 \leq j \leq n, 
\]
where $d \geq 1$ is a parameter. 

Intuitively, increasing the value of $d$ makes the zero entries of $M$ more important in the objective function, which leads them to be approximated by small values. This observation will be used to show that, for $d$ sufficiently large, the optimal value $p^*$ of \eqref{WLRA1} will be close to $|E|-|E^*|$, the optimal value of \eqref{MBP} (Lemma~\ref{optWLRA}). 

A maximal biclique in $G_b$ is a biclique not contained in a larger biclique, and can be seen as a `locally' optimal solutions of  \eqref{MBP}. We will show that, as the value of parameter $d$ increases, the local minima of \eqref{WLRA1} get closer to binary vectors describing maximal bicliques in $G_b$.

Example~\ref{examp} illustrates the situation: the graph  $G_b$ corresponding to matrix $M$ (cf.\@ Figure~\ref{graph}) contains four maximal bicliques $\{s_1,s_3,t_1,t_3\}$, $\{s_2,s_3,t_2,t_3\}$, $\{s_3,t_1,t_2,t_3\}$ and  $\{s_1,s_2,s_3,t_3\}$, and the weight matrix $W$ that was used is similar to the case $d=100$ in problem~\eqref{WLRA1}. We now observe that \eqref{WLRA1} has four local optimal solutions as well (cf.\@ Figure~\ref{localmin}) close to $(\frac{\sqrt{2}}{2},0)$, $(0,\frac{\sqrt{2}}{2})$, $(0,0)$ and $(\frac{\sqrt{2}}{2},\frac{\sqrt{2}}{2})$. There is a one to one correspondence between these solutions and the four maximal bicliques listed above (in this order). 
For example, for $(x,y)=(\frac{\sqrt{2}}{2},0)$ we have $u(x,y) = (\frac{\sqrt{2}}{2}, 0, \frac{\sqrt{2}}{2})^T$, 
 $v^*(x,y)$ is approximately equal to $({\sqrt{2}}, 0, {\sqrt{2}})^T$, and this solution corresponds to the maximal biclique $\{s_1,s_3,t_1,t_3\}$. 
 
Notice that a similar idea was used  in \cite{GG08} to prove NP-hardness of the rank-one nonnegative factorization problem $\min_{u \in \mathbb{R}_+^m, v \in \mathbb{R}_+^n} ||M-uv^T||_F$, where the zero entries of $M$ were replaced by sufficiently large negative ones. 
\begin{remark}[Link with classical quadratic penalty method]  
It is worth noting that \eqref{WLRA1} can be viewed as the application of the classical quadratic {penalty approach} to the biclique problem, see, e.g., \cite[\S 17.1]{NR06}. In fact, defining $F = \{ (i,j) | M_{ij} = 1 \}$ and its complement $\bar{F} = \{ (i,j) | M_{ij} = 0 \}$, the biclique problem can be formulated as 
\begin{equation} \label{bicky}
\min_{u, v} \sum_{(i,j) \in F} (1-u_iv_j)^2 \;\; \text{ such that } u_iv_j = 0 \text{ for all } (i,j) \in \bar{F}. 
\end{equation}
Indeed, in this formulation, it is clear that any optimal solution can be chosen such that vectors $u$ and $v$ are binary, from which the equivalence with problem \eqref{MBP} easily follows.  Penalizing (quadratically) the equality constraints in the objective, we obtain 
 \[
 P_d(u,v) = \sum_{(i,j) \in F} (1-u_iv_j)^2 +  d \sum_{(i,j) \in \bar{F}} (u_iv_j)^2, 
\]
where $d \geq 0$ is the penalty parameter. We now observe that our choice of $W$ at the beginning of this section gives $P_d(u,v) = ||M-uv^T||_W^2$, i.e.,  \eqref{WLRA1} is exactly equivalent to minimizing $P_d(u,v)$. This implies that, as $d$ grows, minimizers of problem \eqref{WLRA1} will tend to solutions of the biclique problem \eqref{MBP}. Our goal is now to prove a more precise statement about the link between these two problems: we provide (in Lemma~\ref{optWLRA}) an explicit value for $d$ that guarantees a small difference between the optimal values of these two problems.
\vspace{0.1cm} 
%
%
\end{remark}

First, we establish that for any $(u,v)$ such that $||M-uv^T||_W^2 \leq |E|$, the absolute value of the row or the column of $uv^T$ corresponding to a zero entry of $M$ must be smaller than a constant inversely proportional to $\sqrt[4]{d}$. 
\begin{lemma} \label{lembic}
Let $(i,j)$ be such that $M_{ij} = 0$, then 
$\forall (u,v)$ such that $||M-uv^T||_W^2 \leq |E|$, 
\[ 
\min\Big( \, \max_{1 \leq k \leq n} |u_i v_k|, \, \max_{1 \leq p \leq m} |u_p v_j| \Big) 
\leq \sqrt[4]{\frac{4|E|^2}{d}}.  
\]
\end{lemma}
\begin{proof} 
Without loss of generality $u$ and $v$ can be scaled such that \@ $||u||_2=||v||_2$ without changing the product $uv^T$, i.e., we replace $u$ by $u' = \sqrt{\frac{||v||_2}{||u||_2}} u$ and $v$ by 
$v' = \sqrt{\frac{||u||_2}{||v||_2}} v$ so that $||u'||_2=||v'||_2 = \sqrt{||u||_2||v||_2}$ and $u'v'^T = uv^T$. 
First, observe that since $||.||_W$ is a norm, 
\[
||uv^T||_W - \sqrt{|E|} = ||uv^T||_W - ||M||_W \leq ||M-uv^T||_W \leq \sqrt{|E|}.
\] 
Since all entries of $W$ are larger than 1 ($d \geq 1$), we have 
\[
||u||_2 ||v||_2 = ||uv^T||_F \leq  ||uv^T||_W \leq \sqrt{4|E|}, 
\] 
and then $||u||_2 = ||v||_2 \leq \sqrt[4]{4|E|}$. 

Moreover $d (0-u_iv_j)^2 \leq ||M-uv^T||_W^2 \leq |E|$, so that $|u_iv_j| \leq \sqrt{\frac{|E|}{d}}$ which implies that either 
$|u_i| \leq \sqrt[4]{\frac{|E|}{d}}$ 
or 
$|v_j| \leq \sqrt[4]{\frac{|E|}{d}}$. 
 Combining the above inequalities with the fact that $(\max_{1 \leq k \leq n} |v_k|)$ and $(\max_{1 \leq p \leq m} |u_p|)$ are bounded above by $||u||_2 = ||v||_2 \leq \sqrt[4]{4|E|}$ completes the proof. \vspace{0.2cm}
\end{proof}

We now prove the following general lemma which, combined with Lemma~\ref{lembic} above, will allow us to derive a lower bound on the objective function of \eqref{WLRA1} (it will also be used for the proof of the problem with missing data in Section~\ref{LRMD}). 


%
%
%
 
\begin{lemma} \label{keylem}  
Let $M \in \{0,1\}^{m \times n}$ be the biadjacency matrix of a bipartite graph $G_b = (V,E)$, $W \in \mathbb{R}_+^{m \times n}$ a weight matrix such that $W_{ij}=1$ for each pair $(i,j)$ satisfying $M_{ij}=1$, and $(u,v)$ be such that 
\begin{equation} \label{keypeqt}
\min\Big( \, \max_{1 \leq k \leq n} |u_i v_k|, \, \max_{1 \leq p \leq m} |u_p v_j| \Big) 
\leq c, 
\end{equation}
for each pair $(i,j)$ satisfying $M_{ij} = 0$, where $0 < c \leq 1$. Let also $p = |E|-|E^*|$ be the optimal objective function value of \eqref{MBP}. Then, if $p > 0$, we have 
\[
||M-uv^T||_W > p (1 - 2c). 
\] 
\end{lemma}
\begin{proof} 
Define the biclique corresponding to the following set $\Omega_{c}(u,v) \subseteq \{1,2,\dots,m\} \times \{1,2,\dots,n\}$ 
\begin{equation} 
\Omega_{c}(u,v) = \{ \,  i \ | \ \exists j \text{ s.t. } |u_iv_j| > c \, \} \times \{ \,  j \ | \ \exists i \text{ s.t. } |u_iv_j| > c \, \}. \nonumber 
\end{equation}
This biclique is part of the original graph, i.e., every edge in $\Omega_c(u,v)$ belongs to $G_b$. Indeed, if $M_{ij}=0$, the pair $(i,j)$ cannot belong to $\Omega_c(u,v)$ since, by Equation~\eqref{keypeqt}, the absolute value of either the $i{}^\textrm{th}$ row or the $j{}^\textrm{th}$ column of $u v^T$ is smaller than $c$. By construction, we also have that the entries $M_{ij}$ corresponding to pairs $(i,j)$ not in  the biclique $\Omega_{c}(u,v)$ are approximated by values smaller than $c$.  The error corresponding to a unit entry of $M_{ij}$ not in the biclique $\Omega_{c}(u,v)$ is then at least $(1-c)^2$ (because the corresponding weight $W_{ij}$ is equal to one). Since there are at least $p = |E|-|E^*|$ such entries (because there are $|E|$ unit entries in $M$ and at most $|E^*|$ pairs in biclique $\Omega_c(u,v)$), we have 
\begin{equation} \nonumber 
||M-uv^T||_W^2 \geq (1-c)^2 p > p (1-2c) = p - 2pc. 
\end{equation}
\end{proof}

We can now provide lower and upper bounds on the optimal value $p^*$ of \eqref{WLRA1}, and show that it is not too different from the optimal value $|E| - |E^*|$ of \eqref{MBP}. 
\begin{lemma} \label{optWLRA}  
Let $0 < \epsilon \leq 1$. For any value of parameter $d$ such that $d \geq \frac{2^6 |E|^6}{\epsilon^4}$, the optimal value $p^*$ of \eqref{WLRA1} satisfies 
\[
|E| - |E^*| - \epsilon < p^* \leq |E| - |E^*|. 
\]
\end{lemma}
\begin{proof} 
Let $(u,v)$ be an optimal solution of \eqref{WLRA1} (since $W > 0$, there always exists at least one optimal solution,  cf.\@ Section~\ref{formu}), and let us note $p = |E|-|E^*| \geq 0$. If $p = 0$, then $p^* = 0$ and the result is trivial (it is the case when the rank of $M$ is one, i.e., $G_b$ contains only one biclique). Otherwise, since any optimal solution of \eqref{MBP} plugged in \eqref{WLRA1} achieves an objective function equal to $p$, we must have 
\[
p^* = ||M-uv^T||_W^2 \leq p = |E| - |E^*|, 
\]
which gives the upper bound. 

Since $d$ is greater than $4|E|^2$ for any $0 < \epsilon \leq 1$, the constant $\alpha = \sqrt[4]{\frac{4|E|^2}{d}}$ appearing in Lemma~\ref{lembic} is smaller than one. This means that  Lemma~\ref{keylem} is applicable, so that we have  
\[
||M-uv^T||_W^2 > p - 2\alpha p \geq p - 2 \alpha |E| \geq p - \epsilon, 
\]
which gives the lower bound (the last inequality follows from the fact that $2 \alpha |E| \leq \epsilon$ is equivalent to the condition $d \geq \frac{2^6|E|^6}{\epsilon^4}$). 
\end{proof}

This result implies that for $\epsilon = 1$, i.e., for $d \geq (2|E|)^6$, 
we have $|E|-|E^*|-1 < p^* \leq |E|-|E^*|$,  
and therefore computing $p^*$ exactly would allow to recover $|E^*|$ (since 
$|E^*| = |E|-\lceil p^* \rceil$), 
which is NP-hard. Since the reduction from \eqref{MBP} to \eqref{WLRA1} is polynomial (it uses the same matrix $M$ and a weight matrix $W$ whose description has polynomial length), we conclude that solving \eqref{WLRA1} exactly is NP-hard.
The next result shows that even solving \eqref{WLRA1} approximately is NP-hard.

\begin{corollary} \label{th1cor}
For any $d \geq (2mn)^6$,  $M \in \{0,1\}^{m \times n}$ and $W \in \{1,d\}^{m \times n}$, it is NP-hard to find an approximate solution of rank-one  \eqref{WLRA} with objective function accuracy less than 
$1-\frac{(2mn)^{3/2}}{d^{1/4}}$. 
\end{corollary}
\begin{proof}
Let $d \geq {(2mn)^6}$, 
$0 < \epsilon = \frac{(2mn)^{3/2}}{d^{1/4}} < 1$, and 
$(\bar{u},\bar{v})$ be an approximate solution of \eqref{WLRA1} with objective function accuracy  $(1-\epsilon)$, i.e., $p^* \leq \bar{p} = ||M-\bar{u}\bar{v}^T||_W^2 \leq p^* + 1 - \epsilon$. 
Since $d = \frac{(2mn)^6}{\epsilon^4} \geq \frac{(2|E|)^6}{\epsilon^4}$, Lemma~\ref{optWLRA} applies and we have   
\[
|E|-|E^*|-\epsilon \; < \; p^* \leq \; \bar{p} \; \leq \; p^* + 1 - \epsilon \; \leq \; |E|-|E^*| + 1 - \epsilon. 
\]
We finally observe that knowing $\bar{p}$ allows to recover $|E^*|$, which is NP-hard. In fact, adding $\epsilon$ to the above inequalities  gives $|E|-|E^*| < \bar{p} + \epsilon \leq  |E|-|E^*| + 1$, 
and therefore 
\[
|E^*| = |E| - \Big\lceil \bar{p} + \epsilon \Big\rceil + 1. 
\] 
\end{proof}

We are now in position to prove Theorem~\ref{WLRAnphard}, which deals with the hardness of rank-one (WLRA) with bounded weights.  

\begin{proof}[\textbf{Theorem~\ref{WLRAnphard}}] 
Let us use Corollary~\ref{th1cor} with $W \in \{1,d\}^{m \times n}$, and define $W' =  \frac{1}{d} W \in \{\frac{1}{d},1\}^{m \times n}$. 
 Clearly, replacing $W$ by $W'$ in \eqref{WLRA1} 
 simply amounts to multiplying the objective function by $\frac{1}{d}$, with $||M-uv^T||_{W'}^2 =  \frac{1}{d} ||M-uv^T||_{W}^2$. 
 Taking $d^{1/4} = 2 (2mn)^{3/2}$ in Corollary~\ref{th1cor}, we obtain that for $M \in \{0,1\}^{m \times n}$ and $W \in ]0,1]^{m \times n}$, it is NP-hard to find an approximate solution of rank-one \eqref{WLRA} with objective function accuracy less than 
$\frac{1}{d} \Big(1-\frac{(2mn)^{3/2}}{d^{1/4}} \Big) = \frac{1}{2d} = 2^{-11} (mn)^{-6}. $
\end{proof}

\begin{remark}
The above bounds on $d$ have been crudely estimated, and can be improved. Our main goal here was to show existence of a polynomial-time reduction from \eqref{MBP} to rank-one \eqref{WLRA}. 
\end{remark}

\begin{remark} \label{Remnmu}
Using the same construction as in \cite[Theorem 3]{GG09}, 
this rank-one NP-hardness result can be generalized to any factorization rank, i.e., approximate \eqref{WLRA} for any fixed rank $r$ is NP-hard. The idea is the following: given a bipartite graph $G_b$ with biadjacency matrix $M \in \{0,1\}^{m \times n}$, we construct a larger bipartite graph $G_b'$  which is made of $r$ disconnected copies of $G_b$, whose biadjacency matrix is therefore given by 
\[
M' =  \left( \begin{array}{cccc}   M & \mathbf{0}_{m \times n} & \dots & \mathbf{0}_{m \times n} \\
\mathbf{0}_{m \times n} & M & & {\mathbf{0}}_{m \times n} \\
\vdots & & \ddots & \vdots \\
\mathbf{0}_{m \times n} & \dots & & M
 \end{array} \right) \in  \{0,1\}^{rm \times rn}. 
\]
Clearly, no biclique in this graph can be larger than a maximum biclique in $G_b$, and there are (at least) $r$ disjoint bicliques with such maximum size in $G'_b$. Letting $(U,V) \in \mathbb{R}^{rm \times r} \times \mathbb{R}^{r \times rn}$ be an optimal solution of the rank-$r$  \eqref{WLRA} problem with $M'$ above and weights $W' = M' + d(\mathbf{1}_{rm \times rn} - M')$ defined as before, 
it can be shown that, for $d$ sufficiently large, each rank-one factor $U_{:k}V_{:k}^T$ must correspond to a maximum biclique of $G_b$.  
\end{remark}

\subsection{Low-Rank Matrix Approximation with Missing Data} \label{LRMD}

The above NP-hardness proof does not cover the case when $W$ is binary, corresponding to missing data in the matrix to be approximated (or to low-rank matrix completion with noise). 
This corresponds to the following problem 
\begin{equation} 
 \inf_{
U  \in \mathbb{R}^{m \times r}, V \in \mathbb{R}^{n \times r}} 
\quad
||M-UV^T||_W^2 = \sum_{ij} W_{ij} (M-UV^T)_{ij}^2\;, \quad W \in \{0,1\}^{m \times n}. 
\nonumber 
\end{equation}
In the same spirit as before, we consider the following rank-one version of the problem 
\begin{equation} \tag{MD-1d}  \label{PCAMD1} 
p^* = \inf_{u \in \mathbb{R}^{m}, v \in \mathbb{R}^{n}}  ||M-uv^T||_{W}^2,
\end{equation}
with input data matrices $M$ and $W$ defined as follows
\[
M = \left( \begin{array}{c|c}
     M_b  &   \mathbf{0}_{s \times Z}  \\  \hline
     \mathbf{0}_{Z \times t}    & d I_{Z}  \\ \end{array} \right)  
     \text{ and }
W = \left( \begin{array}{c|c}
     \mathbf{1}_{s \times t}  & B_1    \\ \hline
     B_2    & I_{Z}  \\ \end{array} \right), 
\]
where $M_b \in \{0,1\}^{s \times t}$ is the biadjacency matrix of the bipartite graph $G_b = (V,E)$, $d > 1$ is a parameter, $Z = st-|E|$ is the number of zero entries in $M_b$, and $m = s+Z$ and $n = t+Z$ are the dimensions of $M$ and $W$.  

Binary matrices $B_1 \in \{0,1\}^{s \times Z}$ and $B_2 \in \{0,1\}^{Z \times t}$ are constructed as follows: assume the $Z$ zero entries of $M_b$ can be enumerated as 
\[
\{M_b(i_1,j_1), M_b(i_2, j_2), \ldots, M_b(i_Z,j_Z)\},
\] 
and let $k_{ij}$ be the (unique) index $k$ ($1 \le k \le Z)$ such that $(i_k,j_k)=(i,j)$ (therefore $k_{ij}$ is only defined for pairs $(i,j)$ such that $M_b(i,j)=0$, and establishes a bijection between these pairs and the set $\{1, 2, \ldots, Z\}$). We now define matrices $B_1$ and $B_2$ as follows: for every index $1 \le k_{ij} \le Z$, we have 
\[ B_1(i,k_{ij}) = 1, B_1(i',k_{ij}) = 0\ \forall i' \neq i \text{ and } B_2(k_{ij},j) = 1, B_2(k_{ij},j') = 0\ \forall j' \neq j \,.\]
Equivalently, each column of $B_1$ (resp.\@ row of $B_2$) corresponds to a different zero entry $M_b(i,j)$, and contains only zeros except for a one at position $i$ within the column (resp.\@ at position $j$ within the row). Hence the matrix $B_1$ (resp.\@ $B_2$) contains only zero entries except $Z$ entries equal to one, one in each column (resp.\@ row).

In the case of Example~\ref{examp}, we get
\[
M = \left( \begin{array}{c|c}
       \begin{array}{ccc} 1 & 0 & 1 \\ 
       0 & 1 & 1 \\ 
       1 & 1 & 1 \end{array} & \mathbf{0}_{3 \times 2}  \\ \hline
      \mathbf{0}_{2 \times 3}	   & d \, I_{2}  \\  
    \end{array} \right) \; \textrm{ and } \;
W = \left( \begin{array}{c|c}
      \mathbf{1}_{3 \times 3} & \begin{array}{cc} 1 & 0 \\ 0 & 1 \\ 0 & 0 \end{array}  \\ \hline
      \begin{array}{ccc} 0 & 1 & 0  \\ 1 & 0 & 0 \\ \end{array}	   & I_{2}  \\  
    \end{array} \right),
\]
i.e., the matrix to be approximated can be represented as 
\begin{equation} \label{matpi}
 \left( \begin{array}{ccc|cc}
1 & 0 & 1 & 0 & ? \\
0 & 1 & 1 & ? & 0 \\
1 & 1 & 1 & ? & ?  \\ \hline
? & 0 & ? & d & ? \\
0 & ? & ? & ? & d 
    \end{array} \right).
\end{equation}

For any feasible solution $(u,v)$ of \eqref{PCAMD1}, we also note
\[
u = \left( \begin{array}{cc}
    u^{(b)}   \\
 	  u^{(d)}    \\
\end{array} \right) \in \mathbb{R}^m, \; u^{(b)} \in \mathbb{R}^s \textrm{ and } u^{(d)} \in \mathbb{R}^{Z},
\]
\[
v = \left( \begin{array}{cc}
    v^{(b)}   \\
 	  v^{(d)}     \\
\end{array} \right) \in \mathbb{R}^n, \; v^{(b)} \in \mathbb{R}^t \textrm{ and } v^{(d)} \in \mathbb{R}^{Z}.
\]

We will show that this formulation ensures that, as $d$ increases, 
the zero entries of matrix $M_b$ (the biadjacency matrix of $G_b$ which appears as the upper left block of matrix $M$) 
have to be approximated with smaller values. Hence, as for \eqref{WLRA1}, we will be able to prove that the optimal value $p^*$ of \eqref{PCAMD1} will have to get close to the minimal value $|E|-|E^*|$ of \eqref{MBP}, implying NP-hardness of its computation.  

Intuitively, when $d$ is large, the lower right matrix $dI_{Z}$ of $M$ will have to be approximated by a matrix with large diagonal entries, since they are weighted by unit entries in matrix $W$. Hence $u^{(d)}_{k_{ij}}v^{(d)}_{k_{ij}}$ has to be large for all $1 \leq k_{ij} \leq Z$. We then have at least either $u^{(d)}_{k_{ij}}$ or $v^{(d)}_{k_{ij}}$ large for all $k_{ij}$ (recall that each $k_{ij}$ corresponds to a zero entry in $M$ at position $(i,j)$, cf.\@ definition of $B_1$ and $B_2$ above).  
By construction, we also have two entries $M(s+k_{ij},j) = 0$ and $M(i,t+k_{ij}) = 0$ with unit  weights corresponding to the nonzero entries $B_1(i,k_{ij})$ and $B_2(k_{ij},j)$, which then also  have to be approximated by small values. If $u^{(d)}_{k_{ij}}$ (resp.\@ $v^{(d)}_{k_{ij}}$) is large, then $v^{(b)}_j$ (resp.\@ $u^{(b)}_i$) will have to be small since $u^{(d)}_{k_{ij}}v^{(b)}_j \approx 0$ (resp.\@ $u^{(b)}_iv^{(d)}_{k_{ij}} \approx 0$). Finally, either $u^{(b)}_i$ or $v^{(b)}_j$ has to be small, implying that $M_b(i,j)$ is approximated by a small value, because $(u^{(b)},v^{(b)})$ can bounded independently of the value of $d$. \\

We now proceed as in Section~\ref{WLRAcomp}. Let us first give an upper bound for the optimal value $p^*$ of \eqref{PCAMD1}. 
\begin{lemma} \label{optp} 
For $d > 1$, the optimal value $p^*$ of \eqref{PCAMD1} is bounded above by $|E|-|E^*|$, i.e., 
\begin{equation} \label{upperbound}
p^* =  \inf_{u \in \mathbb{R}^{m}, v \in \mathbb{R}^{n}}  ||M-uv^T||_{W}^2 \leq |E|-|E^*|.
\end{equation}
\end{lemma}
\begin{proof}
Let us build the following feasible solution $(u,v)$ of \eqref{PCAMD1}: $(u^{(b)},v^{(b)})$ is a (binary) optimal solution of \eqref{MBP} and $(u^{(d)},v^{(d)})$ is defined as\footnote{Notice that this construction is not symmetric, and the variant  using $u^{(b)}$ instead of $v^{(b)}$ to define $u^{(d)}$ and $v^{(d)}$  is also possible.} 
\begin{equation} \label{uvkij}
u^{(d)}_{k_{ij}} = \left\{ 
\begin{array}{ccc} 
d^K & \textrm{ if } & v^{(b)}_j = 0, \\
d^{1-K} & \textrm{ if } & v^{(b)}_j = 1, \\
\end{array} 
\right.
\textrm{ and } \;
v^{(d)}_{k_{ij}} = \left\{ 
\begin{array}{ccc} 
d^{1-K} & \textrm{ if } & v^{(b)}_j = 0, \\
d^{K} & \textrm{ if } & v^{(b)}_j = 1, \\
\end{array} 
\right.
\end{equation}
where $K$ is a real parameter and $k_{ij}$ is the index of the column of $B_1$ and the row of $B_2$ corresponding to the zero entry $(i,j)$ of $M_b$ (i.e., $(i,j)=(i_{k_{ij}}, j_{k_{ij}})$). 

We have that 
\[
(u{v}^T) \circ W = \left( \begin{array}{cc}
     u^{(b)} {v^{(b)}}^T  & D_1 \vspace{0.1cm}   \\
     D_2    & d I_{Z}  \\ \end{array} \right), 
\]
where $\circ$ is the component-wise (or Hadamard) product between two matrices, and matrices $D_1$ and $D_2$ satisfy 
\[
\left\{ 
\begin{array}{cccc} 
D_i(l,p)  & = & 0  & \textrm{ if }  B_i(l,p) = 0, \\
D_i(l,p) & \in & \{0, d^{1-K}\} &  \textrm{ if }  B_i(l,p) = 1, \\
\end{array} 
\right. \quad i = 1, 2.
\] 
In fact, let us analyze the four blocks of $(u{v}^T) \circ W$: 
\begin{enumerate}
\item Upper-left: the upper-left block of $W$ and $uv^T$ are respectively the all-one matrix and $u^{(b)} {v^{(b)}}^T$. 
\item Lower-right: since the lower-right block of $W$ is the identity matrix, we only need to consider the diagonal entries of the lower-right block of $uv^T$, which are given by  $u^{(d)}_{k_{ij}} v^{(d)}_{k_{ij}} = d$ for $k_{ij}= 1, 2, \dots, Z$, cf.\@ Equation~\eqref{uvkij}. 
\item Upper-right and lower-left: by definition of $B_1$ and $B_2$, the only entries of $D_1$ and $D_2$ which may be different from zero are given by 
\[
D_1(i,k_{ij}) = u^{(b)}_i v^{(d)}_{k_{ij}} 
\; \text{ and } \; 
D_2(k_{ij},j) = u^{(d)}_{k_{ij}} v^{(b)}_j, 
\] 
for all $(i,j)$ such that  $M_b(i,j) = 0$. By construction, we have either $v^{(b)}_j = 0$ or $v^{(b)}_j = 1$.  
If $v^{(b)}_j = 0$, 
then $v^{(d)}_{k_{ij}} = d^{1-K}$ by Equation~\eqref{uvkij} and 
we  have  $D_1(i,k_{ij}) = u^{(b)}_i v^{(d)}_{k_{ij}} \in \{0, d^{1-K}\}$ and $D_2(k_{ij},j) = 0$. 
If $v^{(b)}_j = 1$, we have  $u^{(b)}_i = 0$ (since $M_b(i,j) = 0$) and $u^{(d)}_{k_{ij}} = d^{1-K}$ by Equation~\eqref{uvkij} whence  $D_1(i,k_{ij}) = 0$ and $D_2(k_{ij},j) = d^{1-K}$. 
\end{enumerate} 
Finally, $D_1$ and $D_2$ have at most $Z$  non-zero entries (recall $Z$ is the number of zero entries in $M_b$), which are all equal to $d^{1-K}$; therefore, 
\begin{equation} \label{Kv}
p^* \leq ||M-u {v}^T||_{W}^2 \leq |E|-|E^*| + 2Z d^{2(1-K)}, \quad \forall K.
\end{equation}
Since $d>1$, taking the limit $K \rightarrow +\infty$ gives the result. 
\end{proof}

\begin{example} \label{ex3}
Let us illustrate the construction of Lemma~\ref{optp} on the matrix from Example~\ref{examp}, which contains two maximum bicliques with 4 edges, including the one corresponding to $u^{(b)} = (0,1,1)^T$ and $v^{(b)} = (0,1,1)^T$. Taking  $u = (0, 1, 1, d^{1-K}, d^K )^T$ and $v = (0, 1, 1, d^K, d^{1-K})^T$, we obtain 
\begin{equation} \nonumber
 \left( \begin{array}{ccc|cc}
1 & 0 & 1 & 0 & ? \\
0 & 1 & 1 & ? & 0 \\
1 & 1 & 1 & ? & ?  \\ \hline
? & 0 & ? & d & ? \\
0 & ? & ? & ? & d 
    \end{array} \right) \approx uv^T = 
     \left( \begin{array}{ccc|cc}
0 & 0 & 0 & 0 & 0 \\
0 & 1 & 1 & d^K & d^{1-K} \\
0 & 1 & 1 & d^K & d^{1-K}  \\ \hline
0 & d^{1-K} & d^{1-K} & d & d^{2(1-K)} \\
0 & d^K & d^K & d^{2K} & d 
    \end{array} \right), 
\end{equation} 
with $||M-uv^T||_W^2 = {(|E|-|E^*|)} + 2 d^{2(1-K)} = 3 + 2 d^{2(1-K)}$, which is less than the bound $3 + 4 d^{2(1-K)}$ guaranteed by Equation~\eqref{Kv}. \vspace{0.1cm} 
\end{example}

We now prove a property similar to Lemma~\ref{lembic} for any solution with objective value smaller that $|E|$. 
\begin{lemma} \label{lembic2} Let $d > \sqrt{|E|}$ and $(i,j)$ be such that $M_b(i,j) = 0$, then the following holds for any pair $(u,v)$ such that $||M-u {v}^T||_{W}^2 \leq |E|$:
\begin{equation} \label{colrowB}
\min\Big( \, \max_{1 \leq k \leq n} |u_i v_k|, \, \max_{1 \leq p \leq m} |u_p v_j| \Big) 
\leq \frac{\sqrt{2} \, |E|^{\frac{3}{4}}}{\big(d-\sqrt{|E|}\big)^{\frac{1}{2}}}. 
\end{equation}
\end{lemma}
\begin{proof} Without loss of generality we set $||u^{(b)}||_2 = ||v^{(b)}||_2$ by scaling $u$ and $v$ without changing $uv^T$. Observing that
\begin{eqnarray*}
||u^{(b)}||_2||v^{(b)}||_2 - \sqrt{|E|} = ||u^{(b)}v^{(b)}{}^T||_F - ||M_b||_F 
& \leq &  ||M_b-u^{(b)}v^{(b)}{}^T||_F \\
& \leq & ||M-u {v}^T||_{W} \leq \sqrt{|E|},
\end{eqnarray*}
we have $||u^{(b)}||_2||v^{(b)}||_2 \leq 2\sqrt{|E|}$, and $||u^{(b)}||_2 =||v^{(b)}||_2 \leq \sqrt{2}|E|^{\frac{1}{4}}$. \\
Assume $M_b(i,j)$ is zero for some pair $(i,j)$ and let $k = k_{ij}$ denote the index of the corresponding column of $B_1$ and row of $B_2$ (i.e., such that $B_1(i,k) = B_2(k,j) = 1$). 
By construction, $u^{(d)}_kv^{(d)}_k$ has to approximate $d$ in the objective function. This implies $(u^{(d)}_kv^{(d)}_k-d)^2 \leq |E|$ and then
\[
u^{(d)}_kv^{(d)}_k \geq d - \sqrt{|E|} > 0.
\]
Suppose $|u^{(d)}_k|$ is greater than $|v^{(d)}_k|$ (the case where $|v^{(d)}_k|$ is greater than $|u^{(d)}_k|$ is similar), which implies $|u^{(d)}_k| \geq (d-|E|^{\frac{1}{2}})^{\frac{1}{2}}$. Moreover, since $B_2(k,j)$ is a unit weight, we have that $u^{(d)}_kv_j$ has to approximate zero in the objective function, implying 
\[
(u^{(d)}_k v_j - 0) ^2 \leq {|E|} \quad \Rightarrow \quad |u^{(d)}_k v_j| \leq \sqrt{|E|}.
\]
Hence 
\begin{equation} \label{vj}
|v_j| \leq \frac{\sqrt{|E|}}{|u^{(d)}_k|} \leq  \frac{|E|^{\frac{1}{2}}}{\big(d-\sqrt{|E|}\big)^{\frac{1}{2}}},
\end{equation}
and since $(\max_{1 \leq p \leq m} |u_p|)$ is bounded by $||u^{(b)}||_2 \leq \sqrt{2}|E|^{\frac{1}{4}}$, the proof is complete. 
\end{proof} \vspace{0.1cm}

Using Lemma~\ref{keylem}, we can now derive a lower bound for the value of $p^*$.  
\begin{lemma} \label{PCAprec} 
Let $0 < \epsilon \leq 1$. For any value of parameter $d$ strictly greater than $\frac{8 |E|^{\frac{7}{2}}}{\epsilon^2} + |E|^{\frac{1}{2}}$, the infimum $p^*$ of \eqref{PCAMD1} satisfies
\[
|E|-|E^*| - \epsilon < p^*. 
\]
\end{lemma} 
\begin{proof}
Let us note $p = |E| - |E^*|$. If $p=0$, the result is trivial since $p^* = 0$. 
Otherwise, suppose $p^* \leq p-\epsilon$ and let $\beta = \frac{\sqrt{2} \, |E|^{\frac{3}{4}}}{\big(d-\sqrt{|E|}\big)^{\frac{1}{2}}}$. First observe that $d > \frac{8 |E|^{\frac{7}{2}}}{\epsilon^2} + |E|^{\frac{1}{2}}$ is equivalent to $2 |E| \beta < \epsilon$.   
Then, by continuity of \eqref{PCAMD1}, for any $\delta$ such that $\delta < \epsilon$, there exists a pair $(u,v)$ such that 
\[
||M_b-u^{(b)}v^{(b)}{}^T||_W^2 \leq ||M-uv^T||_{W}^2 \leq p-\delta \leq |E|. 
\]
In particular, let us take $\delta = 2 |E| \beta  < \epsilon$. Observe that $\beta \leq 1$ as soon as $d \geq 2 |E|^{\frac{3}{2}}+|E|^{\frac{1}{2}}$ (which is guaranteed because $0 < \epsilon \leq 1$). 
By Lemma~\ref{lembic2} and Lemma~\ref{keylem} (applied on matrix $M_b$ and the solution $(u^{(b)},v^{(b)})$), we then have 
\[
p - 2\beta p <  ||M_b-u^{(b)}v^{(b)}{}^T||_W^2 \leq ||M-uv^T||_W^2 \leq p-\delta. 
\]
Dividing the above inequalities by $p > 0$, we obtain 
\[
1-2\beta <  1-\frac{\delta}{p} < 1-\frac{\delta}{|E|} \Rightarrow \delta < 2 |E| \beta, 
\]
a contradiction. 
\end{proof}

\begin{corollary} \label{corth2}
For any $d > {8 (mn)^{7/2}} + \sqrt{mn}$, $M \in \{0,1,d\}^{m \times n}$, and $W \in \{0,1\}^{m \times n}$, it is NP-hard to find an approximate solution of rank-one \eqref{WLRA} with  objective function  accuracy $1-\frac{2\sqrt{2} (mn)^{7/4}}{(d-\sqrt{mn})^{1/2}}$. 
\end{corollary}
\begin{proof}
Let $d > 8(mn)^{7/2} + \sqrt{mn}$,   
$0 < \epsilon = \frac{2\sqrt{2} (mn)^{7/4}}{(d - \sqrt{mn})^{1/2}} < 1$, and 
$(\bar{u},\bar{v})$ be an approximate solution of \eqref{WLRA1} with absolute error $(1-\epsilon)$, i.e., $p^* \leq \bar{p} = ||M-\bar{u}\bar{v}^T||_W^2 \leq p^* + 1 - \epsilon$. 
Lemma~\ref{PCAprec} applies because $d = \frac{8 (mn)^{7/2}}{\epsilon^2} + \sqrt{mn} \geq \frac{8 (st)^{7/2}}{\epsilon^2} + \sqrt{st} \geq \frac{8 |E|^{7/2}}{\epsilon^2} + |E|^{1/2}$. Using Lemmas~\ref{optp} and \ref{PCAprec}, the rest of the proof is identical as the one of  Theorem~\ref{WLRAnphard}.  Since the reduction from \eqref{MBP} to \eqref{PCAMD1} is polynomial (description of matrices $W$ and $M$ has polynomial length, since the increase in matrix dimensions from $M_b$ to $M$ is polynomial), we conclude that finding such an approximate solution for \eqref{PCAMD1} is NP-hard. 
\end{proof}

We can now easily derive Theorem~\ref{LRAMDnphard}, which deals with the hardness of rank-one (WLRA) with a bounded matrix $M$.  

\begin{proof}[\textbf{Theorem~\ref{LRAMDnphard}}] 
Replacing $M$ by $M' = \frac{1}{d} M$ in \eqref{PCAMD1} gives an equivalent problem with objective function multiplied by $\frac{1}{d^2}$, since 
$\frac{1}{d^2} ||M-uv^T||_W^2 = ||M'-\frac{uv^T}{d}||_W^2$. 
Taking $d =  2^5 (mn)^{7/2} +\sqrt{mn}$ in Corollary~\ref{corth2}, we find that it is NP-hard to compute an approximate solution of rank-one \eqref{WLRA} for $M \in [0,1]^{m \times n}$ and $W \in \{0,1\}^{m \times n}$, and  with objective function accuracy less than 
$\frac{1}{d^2} \Big(1-\frac{2\sqrt{2} (mn)^{7/4}}{(d-\sqrt{mn})^{1/2}} \Big) 
= \frac{1}{2 d^2} \geq 2^{-12} (mn)^{-7}$. 
\end{proof}

\section{Concluding Remarks}

In this paper, we have studied the complexity of the weighted low-rank approximation  problem (WLRA), and proved that computing an approximate solution with some prescribed accuracy is NP-hard, already in the rank-one case,  both for positive and binary weights (the latter also corresponding to low-rank matrix completion with noise, or PCA with missing data).  \\

The following more general problem is sometimes also referred to as WLRA: 
\begin{equation} \label{gWLRA}
\inf_{U  \in \mathbb{R}^{m \times r}, V \in \mathbb{R}^{r \times n}} \quad
||M-UV||_{(P)}^2, 
\end{equation} 
where $||A||_{(P)}^2 = \text{vec}(A)^T P \text{vec}(A)$, with $\text{vec}(A)$ a vectorization of matrix $A$ and $P$ an $mn$-by-$mn$ positive semidefinite matrix, see \cite{S06} and the references therein. Since our WLRA formulation corresponds to the special case of a diagonal (nonnegative) $P$, our hardness results also apply to Problem~\eqref{gWLRA}.

It is also worth pointing out that, when the data matrix $M$ is nonnegative,  any optimal solution to rank-one \eqref{WLRA} can be assumed to be nonnegative (see discussion for Example~\ref{examp}). 
Therefore, all the complexity results of this paper apply to the weighted nonnegative matrix factorization problem (weighted NMF), 
 which is the following low-rank matrix approximation problem with nonnegativity constraints on the factors 
\begin{equation} \nonumber 
\min_{U \in \mathbb{R}^{m \times r}, V \in \mathbb{R}^{n \times r}}    ||M-UV^T||_W^2  \quad \text{ such that } \quad  U \geq 0, \; V \geq 0.
\end{equation}
Hence, it it is NP-hard to find an approximate solution to \emph{rank-one weighted NMF} (used, e.g., in image processing \cite[Chapter 6]{diep}) and to \emph{rank-one NMF with missing data} (used, e.g.,  for collaborative filtering \cite{CWZ09}). This is in contrast with unweighted rank-one NMF, which is polynomially solvable (e.g., taking the absolute value of the first rank-one factor generated by the singular value decomposition). 
 Note that (unweighted) NMF has been shown to be NP-hard when $r$ is not fixed \cite{Vav} (i.e., when $r$ is part of the input). 

Nevertheless, many questions remain open, including the following: 
\begin{itemize}

\item Our approximation results are rather weak. In fact, they require the objective function accuracy to \emph{increase} with the dimensions of the input matrix, in proportion with $(mn)^{-6}$, which is somewhat counter-intuitive. The reason is twofold: first, independently of the size of the matrix, we needed the objective function value of approximate solutions of problems \eqref{WLRA1} and \eqref{PCAMD1} to be no larger than the objective function of the optimal biclique solution plus one (in order to obtain $|E^*|$ by rounding). Second, parameter $d$ in problems \eqref{WLRA1} and \eqref{PCAMD1} depends on the dimensions of matrix $M$. Therefore, when matrices $W$ or $M$ are rescaled between 0 and 1, the objective function accuracy is affected by parameter $d$, and hence decreases with the dimensions of matrix $M$. Strengthening of these bounds is a topic for further research. 
\item Moreover, as pointed out to us, these results say nothing about the hardness of approximation within a constant multiplicative factor. It would then be interesting to combine our reductions with inapproximability results for the biclique problem (which have yet to be investigated thoroughly, see, e.g., \cite{T08}), or construct reductions from other problems.  
 
\item When $W$ is the matrix of all ones, WLRA can be solved in polynomial-time. We have shown that, when the ratio between the largest and the smallest entry in $W$ is large enough, the problem is NP-hard (Theorem~\ref{WLRAnphard}). It would be interesting to investigate the gap between these two facts, i.e., what is the minimum ratio between the entries of $W$ that leads to an NP-hard WLRA problem? 

\item When $\rank(W) = 1$, WLRA can be solved in polynomial-time (cf.\@ Section~\ref{formu}) while it is NP-hard for a general matrix $W$ (with rank up to $\min(m,n)$). What is the complexity of \eqref{WLRA} if the rank of the weight matrix $W$ is fixed and greater than one, e.g., if $\rank(W) = 2$? 

\item When data is missing, the rank-one matrix approximation problem is NP-hard in general. Nevertheless, it has been observed \cite{CP09} that when the given entries are sufficiently numerous, well-distributed in the matrix, and affected by a relatively low level of noise, the original uncorrupted low-rank matrix can be recovered accurately, with a technique based on convex optimization (minimization of the nuclear norm of the approximation, which can be cast as a semidefinite program). It would then be particularly interesting to analyze the complexity of the problem given additional assumptions on the data matrix, for example on the noise distribution, and deal in particular with situations related to applications. 

\end{itemize}

\subsection*{Acknowledgments}

We thank Chia-Tche Chang for his helpful comments. We are grateful to the insightful comments of the three anonymous reviewers which helped to improve the paper substantially.

\small 

\bibliographystyle{siam}
\bibliography{wlra}

\begin{thebibliography}{10}

\bibitem{CP09}
{\sc E.J. Cand\`{e}s and Y.~Plan}, {\em {Matrix Completion with Noise}}, in
  Proceedings of the IEEE, 2009.

\bibitem{CR09}
{\sc E.J. Cand\`{e}s and B.~Recht}, {\em {Exact Matrix Completion via Convex
  Optimization}}, Foundations of Computational Mathematics, 9 (2009),
  pp.~717--772.

\bibitem{CWZ09}
{\sc F.~Chen, G.~Wanga and C.~Zhang}, {\em {Collaborative filtering using
  orthogonal nonnegative matrix tri-factorization}}, Information Processing \&
  Management, 45(3) (2009), pp.~368--379.

\bibitem{C08}
{\sc P.~Chen}, {\em {Optimization Algorithms on Subspaces: Revisiting Missing
  Data Problem in Low-Rank Matrix}}, International Journal of Computer Vision,
  80(1) (2008), pp.~125--142.

\bibitem{CG84}
{\sc A.L. Chistov and D.Yu. Grigoriev}, {\em {Complexity of quantifier
  elimination in the theory of algebraically closed fields}}, Proceedings of
  the 11th Symposium on Mathematical Foundations of Computer Science, Lecture
  Notes in Computer Science, Springer, 176 (1984), pp.~17--31.

\bibitem{C94}
{\sc P.~Comon}, {\em {Independent component analysis, A new concept?}}, Signal
  Processing, 36 (1994), pp.~287--314.

\bibitem{AE07}
{\sc A.~d'Aspremont, L.~El~Ghaoui, M.I. Jordan, and G.R.G. Lanckriet}, {\em {A
  Direct Formulation for Sparse PCA Using Semidefinite Programming}}, SIAM
  Rev., 49(3) (2007), pp.~434--448.

\bibitem{GZ79}
{\sc K.R. Gabriel and S.~Zamir}, {\em {Lower Rank Approximation of Matrices by
  Least Squares With Any Choice of Weights}}, Technometrics, 21(4) (1979),
  pp.~489--498.

\bibitem{Gar}
{\sc M.R. Garey and D.S. Johnson}, {\em Computers and Intractability: A guide
  to the theory of NP-completeness}, Freeman, San Francisco, 1979.

\bibitem{GG08}
{\sc N.~Gillis and F.~Glineur}, {\em {Nonnegative Factorization and The Maximum
  Edge Biclique Problem}}.
\newblock CORE Discussion paper 2008/64, 2008.

\bibitem{GG09}
\leavevmode\vrule height 2pt depth -1.6pt width 23pt, {\em {Using
  underapproximations for sparse nonnegative matrix factorization}}, Pattern
  Recognition, 43(4) (2010), pp.~1676--1687.

\bibitem{Gol}
{\sc G.H. Golub and C.F. Van~Loan}, {\em {Matrix Computation, 3rd Edition}},
  The Johns Hopkins University Press Baltimore, 1996.

\bibitem{GM98}
{\sc B.~Grung and R.~Manne}, {\em {Missing values in principal component
  analysis}}, Chemom. and Intell. Lab. Syst., 42 (1998), pp.~125--139.

\bibitem{diep}
{\sc N.-D. Ho}, {\em Nonnegative Matrix Factorization - Algorithms and
  Applications}, PhD thesis, Universit\'{e} catholique de Louvain, 2008.

\bibitem{J01}
{\sc D.~Jacobs}, {\em {Linear fitting with missing data for
  structure-from-motion}}, Vision and Image Understanding, 82 (2001),
  pp.~57--81.

\bibitem{J86}
{\sc I.T. Jolliffe}, {\em Principal Component Analysis}, Springer-Verlag, 1986.

\bibitem{KBV09}
{\sc Y.~Koren, R.~Bell, and C.~Volinsky}, {\em {Matrix Factorization Techniques
  for Recommender Systems}}, IEEE Computer, 42(8) (2009), pp.~30--37.

\bibitem{LS1}
{\sc D.D. Lee and H.S. Seung}, {\em {Learning the Parts of Objects by
  Nonnegative Matrix Factorization}}, Nature, 401 (1999), pp.~788--791.

\bibitem{LPW97}
{\sc W.-S. Lu, S.-C. Pei, and P.-H. Wang}, {\em {Weighted low-rank
  approximation of general complex matrices and its application in the design
  of 2-D digital filters}}, IEEE Trans. Circuits Syst. I, 44 (1997),
  pp.~650--655.

\bibitem{MN09}
{\sc I.~Markovsky and M.~Niranjan}, {\em {Approximate low-rank factorization
  with structured factors}}, Computational Statistics \& Data Analysis, 54
  (2010), pp.~3411--3420.

\bibitem{NKS07}
{\sc D.~Nister, F.~Kahl, and H.~Stewenius}, {\em {Structure from Motion with
  Missing Data is NP-Hard}}, in IEEE 11th Int. Conf. on Computer Vision, 2007.

\bibitem{NR06}
{\sc J.~Nocedal and S.J. Wright}, {\em Numerical Optimization, Second Edition},
  Springer, New York, 2006.

\bibitem{Peet}
{\sc R.~Peeters}, {\em {The maximum edge biclique problem is NP-complete}},
  Discrete Applied Mathematics, 131(3) (2003), pp.~651--654.

\bibitem{Sar}
{\sc B.M. Sarwar, G.~Karypis, J.A. Konstan, and J.~Riedl}, {\em {Item-Based
  Collaborative Filtering Recommendation Algorithms}}, in 10th International
  WorldWideWeb Conference, 2001.

\bibitem{S06}
{\sc M.~Schuermans}, {\em Weighted Low Rank Approximation: Algorithms and
  Applications}, PhD thesis, Katholieke Universiteit Leuven, 2006.

\bibitem{SIR95}
{\sc H.~Shum, K.~Ikeuchi, and R.~Reddy}, {\em {Principal component analysis
  with missing data and its application to polyhedral object modeling}}, IEEE
  Trans. Pattern Anal. Mach. Intelligence, 17(9) (1995), pp.~854--867.

\bibitem{SJ04}
{\sc N.~Srebro and T.~Jaakkola}, {\em {Weighted Low-Rank Approximations}}, in
  20th ICML Conference Proceedings, 2004.

\bibitem{T08}
{\sc J.~Tan}, {\em {Inapproximability of maximum weighted edge biclique and its
  applications}}, in Proceedings of the 5th Int.\@ Conf.\@ on Theory and Appl.
  of Models of Comput., 2008.

\bibitem{Vav}
{\sc S.A. Vavasis}, {\em On the complexity of nonnegative matrix
  factorization}, SIAM J.\@ on Optimization, 20 (2009), pp.~1364--1377.

\end{thebibliography}

\end{document}